\documentclass{aupl}  

 \usepackage{
 amsfonts,
 latexsym,
 amssymb,
 amsbsy,
 enumerate,
 mathrsfs,
 nicefrac,
color}

 \newcommand{\labbel}{\label}

 \newtheorem{theorem}{Theorem}[section]
 \newtheorem{lemma}[theorem]{Lemma}

 \newtheorem{proposition}[theorem]{Proposition}
 
 \newtheorem{corollary}[theorem]{Corollary}

 \newtheorem*{theorem*}{Theorem}
 \newtheorem*{corollary*}{Corollary}
 \newtheorem*{proposition*}{Proposition}

 \theoremstyle{definition}
 \newtheorem{definition}[theorem]{Definition}

 \newtheorem*{definition*}{Definition}

 \theoremstyle{remark}

 \newtheorem*{disclaimer*}{Disclaimer}
 \newtheorem*{acknowledgement}{Acknowledgement}

 \DeclareMathOperator{\slim}{slim}

\DeclareMathOperator{\sep}{sep}

 \newcommand{\conc}{^{\smallfrown}}

\begin{document}
 \title{A clean way to separate sets of surreals}

 \author{Paolo Lipparini}
 \address{Dipartimento di Matematica\\Viale della Ricerca Scientifica\\
 Universit\`a di Roma Tor Segnata  \\I-00133 ROME ITALY}
 \urladdr{http://www.mat.uniroma2.it/\textasciitilde lipparin}



 \thanks{Work partially supported by 
PRIN 2012 ``Logica, Modelli e Insiemi''.
Work  performed under the auspices of G.N.S.A.G.A.}

 \begin{abstract}
Let surreal numbers be defined by means of sign sequences.
We give a proof that  if $S < T$ are sets of surreals, then there is some 
surreal $w$
such that $S < w < T$.   
The classical proof is simplified
by observing that, 
for every set  $S$  of surreals,
there exists a surreal $s$ such that, 
for every surreal $w$, we have 
 $S<w$ if and only if    
the restriction of $w$
to the length of $s$  is $ \geq s$.
Hence $S < w < T$
if and only if 
 $w$  satisfies the above
condition, as well as its symmetrical version
with respect to $T$.
It is now enough to check that if $S < T$,
then the two conditions are compatible. 
 \end{abstract}

 \maketitle

\section{Introduction} \labbel{intro} 

We prove the following theorems (the reader who does not
know the definitions will find full
details  below). 

\begin{theorem} \labbel{1}
If $S$ is any set of surreals,
then there exists a surreal $s$ such that, 
for every surreal $z$, the following conditions are equivalent.
 \begin{enumerate}  
 \item
$u<z$, for every $u \in S$;  
\item
the restriction of $z$ to the length of $s$ is
$\geq s$.  
   \end{enumerate}  
 \end{theorem}

\begin{theorem} \labbel{2}
If $S$ and $T$ are sets of surreals and 
$u< v $, for every $u \in S$ and $v \in T$,
then there exists a surreal $w$ such that
  $u< w < v $, for every $u \in S$ and $v \in T$.
 \end{theorem}  

Theorems \ref{1} and \ref{2} follow from, respectively, Theorems \ref{1b}
and \ref{ult}, which shall be proved below.

The present note is intended to be as much self-contained
as possible. The only prerequisite is a basic knowledge of ordinal numbers.
The main facts about ordinals are stated without
proof in Subsection \ref{ord} below. 
Complete details for the construction of
 surreals  as sign sequences
are given in Subsection \ref{surr}.
Subsection \ref{lemin}
contains the definition of an initial segment 
of a surreal, and a useful lemma about it.
Section \ref{separat} contains a brief discussion
about the usefulness of separating sets of surreals,
as well as a few comments about the subsequent proofs.  
Strengthenings   of 
Theorems \ref{1} and  \ref{2}  will be proved in Sections
\ref{bound} and \ref{separ}, respectively.
In Section \ref{separ}  the shortest surreal $w$ satisfying Theorem \ref{2}
will be also evaluated. 
Section \ref{rmksec} contains a few further remarks.

\section{Preliminaries} \labbel{prel} 

\subsection{Ordinals} \labbel{ord} 
Recall that a linearly ordered set
$(W, <)$ is  \emph{well-ordered} if every nonempty subset of $W$ 
has a minimum.
If $w \in W$, the subset 
$W _{<w} = \{v \in W  \mid v < w \}  $
of $W$ is called a (proper) \emph{initial segment}
of $W$ and inherits from $W$ the structure of a well-ordered set.   
We shall usually consider $W$ itself as an (improper)
initial segment of $W$.

It turns out that any two well-ordered sets are comparable, in the sense
that they are either isomorphic, or one is isomorphic
to some proper initial segment of the other. Exactly one of the above possibilities
occurs. One can choose a representative for any isomorphism class
of well-ordered sets. These representatives are called \emph{ordinals}
and can be chosen in such a way that each ordinal is the set of all
smaller ordinals and the order relation $<$ 
coincides with the membership 
relation $\in$. That is, if $ \beta \in \alpha$ and $\alpha$ is an  ordinal,
then $\beta$ is an ordinal which is an initial segment of $\alpha$
and, conversely, if $\beta$ is an ordinal which is a proper initial segment of 
the ordinal $\alpha$, then $\beta \in \alpha $. 
Ordinals shall be always denoted by small Greek  letters
such as $\alpha$, $\beta$, $\gamma \dots$ and any such symbol is 
always intended to represent an ordinal, even when not explicitly
mentioned.  
The exact definition of an ordinal shall not be relevant here.
We shall make heavy use of the assumption that 
an ordinal is well-ordered.  
For every ordinal $\alpha$,
there is the smallest ordinal $\beta$ which is strictly larger than $\alpha$.
Such a $\beta$ is denoted by $\alpha+1$
(one can actually define a sum between ordinals, but we shall
not need it here).
An ordinal $\beta$ is a \emph{successor ordinal}
if it has the form $\alpha+1$, for some ordinal $\alpha$.
  Otherwise, $\beta$ is a \emph{limit ordinal}.
According to the above definition,
the smallest ordinal $0$
(the order-type of the empty order) 
 is considered a limit ordinal.\footnote{\ 
In some cases, it is convenient to consider $0$
to be neither limit nor successor,
that is, the only one of his kind; however,
here it will be more convenient to consider $0$
as a limit ordinal.}
  
Full details about ordinals can be found in Jech \cite{J} or in any 
textbook on set theory.  

\subsection{Surreals} \labbel{surr} 
A \emph{surreal} is a function $s$ from some 
ordinal  $\ell (s)$  to the set 
$\{ +, - \}$. The ordinal $\ell(s)$ 
 depends on $s$ 
and is called the \emph{length} of
$s$.
Notice that we allow 
$\ell(s)=0$; in this case $s$ is the empty sequence. 
The exact nature of the symbols
$+$ and $-$ does not concern us,
one can take any two distinct elements, say,
$+=(0, \emptyset  )$ and $-=(\emptyset , 0 )$.
Surreal numbers have been introduced by
Conway \cite{C}. See Gonshor \cite{G} for a
full  presentation of
the surreals considered, as above, as sequences
of $+$'s and  $-$'s. 
 Notice that  Conway introduced surreals in a different
fashion.
The above way of representing a surreal
 is called its \emph{sign sequence} 
and is also due to Conway.
 Siegel \cite{S} is another useful reference
about surreals.

Recall that an ordinal is equal to the set
of all smaller ordinals, thus 
$\ell(s)$ is the first ordinal
for which $s$ is not defined. 
If $\gamma \geq \ell (s)$ is an ordinal, we  say
that $s(\gamma)$ is \emph{undefined}. 
If $t$ is another surreal, then, with some abuse
of terminology, we write 
$s( \gamma ) =  t( \gamma )$
to mean that  
$s( \gamma ) $ and $ t( \gamma )$
are either both undefined, 
or  both defined and equal. 
Thus
$s( \gamma ) \neq t( \gamma )$
 means that either  
$s( \gamma ) $ and $ t( \gamma )$
are both defined and are distinct, or
 exactly one between 
$s( \gamma ) $ and $ t( \gamma )$
is defined.
In particular $s$ and  $t$ are distinct surreals if and only if
there is some ordinal  
$\gamma$ such that 
$s( \gamma ) \neq t( \gamma )$
in the above sense\footnote{\ Strictly formally, it would be inappropriate 
to consider a surreal as a function
from the class of all ordinals to a set with three
elements 
$\{ +, undefined, \allowbreak  - \}$.
First, we would like a surreal to be a set, rather
than a proper class. 
The above difficulty could be overcome some way.
The point is that if $s( \gamma )$
is undefined, then 
$s( \delta )$
is undefined, too, for every 
$\delta \geq \gamma $.
While it might be intuitively simpler to think of
a surreal as a function to  
 $\{ +, undefined, - \}$, 
it should be clear that $undefined$
plays  a very special role: everything should be 
$undefined$ after the first occurrence of $undefined$.
In spite of the above considerations, we shall be quite informal 
and we shall write $s( \gamma )$, or speak of the value
assumed by $s$ at $\gamma$ even when 
$\gamma \geq \ell (s)$.
When we want $s( \gamma ) $ to be either $+$ or $-$,
we shall explicitly say that   $s( \gamma ) $ is \emph{defined}.}.

We write $s < t$ 
to mean that $s \neq t$ and  
$s( \gamma ) < t( \gamma )$,
where $\gamma$ is the least ordinal such that 
 $s( \gamma ) \neq t( \gamma )$
and with the provision that 
$- < undefined < t$. 
Without the above conventions and more formally,
$s < t$  means that either

(i) there is some ordinal $\gamma$ such that 
 $\gamma < \ell(s)$, 
 $\gamma < \ell(t)$,
$s(\gamma) \neq t( \gamma )$ and
if $\gamma$ is the least such ordinal, 
then 
$s(\gamma) = - $ and $  t( \gamma )=+$, or

(ii)  $\ell(s) < \ell(t)$ 
and $t(\ell(s))=+$, that is,
$t$ is longer than $s$ and 
$t$ assumes the value $+$ at   
the first place at which $s$ is undefined, 
or 

(iii) 
$\ell(s) > \ell(t)$ 
and $s(\ell(t))=-$, that is,
$s$ is longer than $t$ and 
$s$ assumes the value $-$ at   
the first place at which $t$ is undefined.

The class of all surreals is linearly ordered by $<$.

\subsection{A useful lemma about initial segments} \labbel{lemin}

If $\gamma \leq  \ell (s)$,
we define the \emph{$\gamma$-initial segment of $s$}, in symbols, 
$ s _{ \restriction  \gamma }  $  as the surreal 
$t$ such that 
$\ell(t) = \gamma $
and $t(\delta) = s( \delta )$, 
for every $\delta< \gamma $.  
Sometimes we shall allow the possibility 
$\gamma > \ell (s)$; in that case 
we put $ s _{ \restriction  \gamma } = s $. 
If 
$ t= s _{ \restriction  \gamma }  $, for some $\gamma$,
 we say that $t$  is an
\emph{initial segment} of $s$ and
that $s$ is a \emph{prolongment} of $t$.  
If $\gamma < \ell (s)$, we say that
the initial segment is \emph{proper}.
Intuitively, 
$t$  is an
initial segment of $s$ 
if $t$ can be ``prolonged'' to $s$ 
by setting further values as defined.   
The \emph{$ \leq \gamma$-initial segment 
$ s _{ \restriction \leq  \gamma } $ of $s$} is
defined to be
$ s _{ \restriction ( \gamma +1)}  $.
In other words, the $ \leq \gamma$-initial segment of $s$ is
the surreal 
$t$ such that 
$\ell(t) \leq \gamma +1 $
and $t(\delta) = s( \delta )$, 
for every $\delta \leq  \gamma $.

\begin{lemma} \labbel{lem}
Let $s$, $t$  and  $u$ be surreals. 
  \begin{enumerate}   
 \item 
Suppose that $s \neq t$ and $\gamma$ is the first ordinal such that 
 $s( \gamma )  \neq t ( \gamma )$.
Then 
$ s _{ \restriction  \gamma } = t _{ \restriction  \gamma }  $.
 Moreover,
both $s( \gamma'  )$ and  $t( \gamma ')$ are defined, for every
$\gamma' < \gamma $.
\item
If $ s \leq t$,
then   
$ s _{ \restriction  \gamma }  \leq t _{ \restriction  \gamma }  $,
for every ordinal $\gamma$.
\item
If $\gamma$ is an ordinal and    
$ s _{ \restriction  \gamma }  <  t _{ \restriction  \gamma }  $,
then $s<t$. 
Actually,
$ s _{ \restriction  \delta  }  <  t _{ \restriction  \delta  }  $,
for every $\delta \geq \gamma $. 
\item 
If $\alpha$ is an ordinal,
$\ell(u) \leq \alpha $, $\ell(s) \leq \alpha $ 
 and
$u _{ \restriction \leq \gamma } \leq  s _{ \restriction \leq \gamma }    $,
for every $\gamma < \alpha $, then   
$u  \leq s    $. 
  \end{enumerate} 
 \end{lemma} 

\begin{proof} 
(1)  
It is trivial from the definitions 
that $ s _{ \restriction  \gamma } = t _{ \restriction  \gamma }  $.
Thus, if by contradiction $s (  \gamma' )$ is undefined, for some
$\gamma' < \gamma $, then also 
$t ( \gamma') = s (  \gamma' )$
 is undefined, hence, by a comment in footnote 2, both 
$s$ and $t$ are undefined at values larger than $\gamma'$, hence they cannot
be different at $\gamma$, a contradiction.  

 (2) and (3)  are trivial from the definitions. 

Notice that if $ s < t$,
then   it is not true that
$ s _{ \restriction  \gamma }  < t _{ \restriction  \gamma }  $,
for every $\gamma$. Actually,
if $\gamma= 0$, then 
 $ s _{ \restriction  \gamma } $ is the empty sequence, 
for every surreal $s$, thus 
$ s _{ \restriction  0} = t _{ \restriction  0} $,
for every pair $s$, $t$ of surreals.

(4) This is true if 
$u  = s    $.
Otherwise, 
there is $\gamma $ such that 
$u( \gamma ) \neq s ( \gamma ) $.
We have $\gamma < \alpha $,
for any such $\gamma$, 
since, $u( \gamma ) $ and $ s ( \gamma ) $ being  different, at least one
of them is defined, and 
$u  $ and $  s    $
have both length $ \leq \alpha$.
Choose $\gamma$ minimal such that 
$u( \gamma ) \neq s ( \gamma ) $.
Then 
$u _{ \restriction \gamma  } = s _{ \restriction \gamma  }   $,
by (1); moreover, $u ( \gamma ) < s ( \gamma )$, by the definition 
of the order on the surreals,
since  $u( \gamma ) \neq s ( \gamma ) $ and 
$u _{ \restriction \leq \gamma } \leq  s _{ \restriction \leq \gamma }    $,
by assumption.
Then, again by the definition 
of the order on the surreals,
$u _{ \restriction \leq \gamma } <   s _{ \restriction \leq \gamma }    $,
that is,
$u _{ \restriction ( \gamma +1)} <   s _{ \restriction ( \gamma +1)}    $.
Since $\gamma< \alpha $,
then $\gamma+1  \leq \alpha $,
hence we get   
$u = u _{ \restriction \alpha } <   s _{ \restriction \alpha  } = s   $
from (3).
\end{proof}

\section{Separating sets of surreals} \labbel{separat} 

If $S$ and $T$ are sets of surreals, 
then $S<T$ means that $s<t$, 
for every $s \in S$ and $t \in T$.
If $S= \{ s\} $ is a singleton, we shall simply write
 $s< T$  and a similar convention applies if  $T = \{ t\} $.
Notice that the shorthand is consistent with the notation 
$s<t$.
If $S <  u < T$, we shall say that 
$u$ \emph{separates} $S$ and $T$
or simply that $u$ is a  \emph{separating element},
when $S$ and $T$ are clear from the context.

We now want to show that if 
$S < T$ are sets, then there actually exists 
some $u$ separating $S$ and $T$.
This is a fundamental theorem, when surreals are defined as above as sign
sequences, since many constructions
(e.~g., the definitions of the sum of surreals and a big deal 
of other functions and operations)  
rely on the existence of such an $u$.
When surreals are defined ``a la Conway'',
the existence of such an $u$ is, in a sense, 
part of the definition itself of the surreals.
However, Theorems like \ref{boh} below are useful also 
in such a framework, since    
they provide a description of some appropriate $u$ 
in terms of the sign expansions of the elements
of $S$ and $T$.

The existence of some $u$
separating $S$ and $T$ is well-known, e.~g.,
\cite[Theorem 2.1]{G}. However, we believe 
that  our proof  is particularly clean.
In the theory of surreals defined as sign sequences
most proofs proceed by considering several cases---
frequently, a very large number of them.
It seems that we have reduced the number of cases 
in the construction 
of a separating element to the minimum.
In fact, the definition of the shortest separating element
$ \hat{ s } _{ \restriction \gamma }   $ in Theorem \ref{thmthm}  
is not given by cases.
In Theorem \ref{ult} a simpler proof of the existence of
some separating element (possibly, not the shortest) is provided;
the definition of $ \hat {\hat{ s }}$ given in that proof is not by cases, either.

It should be pointed out, however, that the constructions
in Theorems \ref{ult} and \ref{thmthm}
rely on Definition \ref{supast}, which furnishes the shortest 
element $s$ satisfying Theorem \ref{1}.     
Definition \ref{supast} does indeed consider two cases.
This  looks quite natural, anyway,
since the  cases to be considered are when $S$ has a 
maximum and when $S$   has no maximum.
The number of cases (i.e., two!) is 
surely small, in comparison  with  the  usual 
habit in the theory of surreals.
There is also the possibility of giving Definition \ref{supast} 
in a uniform way not involving a division into cases.
See Remark (e) in Section \ref{rmksec}.
All the remaining parts of the proofs in the present note
essentially  rely only 
on the trivial monotonicity properties of restrictions, as
stated in Lemma \ref{lem}.

Finally, let us notice  two curious facts. Though, of course,
we want the strict inequalities $s < u < t$,
for every $t \in  T$ 
and $s \in S$, the proofs eventually 
involve non-strict inequalities expressed in terms of  $\leq$.  
This turns out to be  not particularly surprising, however,
after one  looks at Theorem \ref{1b} below.
Moreover, our proofs in Section \ref{separ} simplify if we consider certain surreals
which turn out to be longer than necessary, sometimes
strictly longer than any other surreal involved.
See Theorems \ref{ult} and \ref{thmthm}.
Without using  this technique, we would be forced 
to resort to the usual divisions by cases.
See Definition \ref{sep} and Theorem \ref{boh}.
The above observation suggests that the technique of 
using prolonged surreals
might have further applications.

\section{A canonical bounding element} \labbel{bound}

We are now going to prove
Theorem \ref{1}. Actually, we shall give an explicit description
of an element $s$ whose existence follows from \ref{1}
and then derive some useful properties of such an $s$.
 
If $u$ is a surreal, let $u \conc +$ be obtained 
by adding a $+$ at the top of  the string $u$.
Formally, if $\ell( u) = \alpha $, 
then $u \conc +$ is the surreal $s$ defined by:  
   $\ell( s ) = \alpha +1 $,
$ s( \alpha ) = +$    
and 
$s ( \gamma ) = u( \gamma )$,
for $\gamma < \alpha $.
The surreal  $u \conc -$ is defined similarly, 
by adding a $-$ at the top.

\begin{definition} \labbel{supast}
For every set $S$ of surreals, 
define $\sup ^* S$  in the following way. 

If $S$ has some maximum
$ u$, let $\sup ^* S = u \conc + $. 

If $S$ has no maximum, 
let $\sup ^* S$ be the surreal $s$ 
defined as follows. 
If $\gamma$ is an ordinal,
let $s(\gamma)$ be defined if and only if 
there is $u \in S$ such that 
$u(\gamma)$ is defined and, for every $u' \geq u$ such that  $u' \in S$,
it happens that  $ u ' _{ \restriction \leq \gamma }
= u _{ \restriction \leq \gamma }  $.
If this is the case, let 
$s( \gamma ) = u ( \gamma )$.
Of course, by the assumption, the definition 
does not depend on the choice of $u$.
Notice also that if $s( \gamma )$ is defined,
then $s(\gamma')$ is defined, for every $\gamma' \leq \gamma $,
hence the definition  actually provides a surreal.

The surreal 
$\inf ^* S$  is defined in the symmetrical way, namely,
if $S$ has some minimum
$ u$, let  
$\inf ^* S = u \conc - $. 
If $S$ has no minimum,  let $\inf ^* S$
be the surreal $t$ such that  
 $t(\gamma)$ is defined if and only if 
there is $v \in T$ such that $v( \gamma )$ is defined and,
for every $v' \leq v$, $v' \in T$,
it happens that  $ v ' _{ \restriction \leq \gamma }
= v _{ \restriction \leq \gamma }  $, and,
if this is the case, then 
let $t( \gamma ) = v ( \gamma )$.
  \end{definition}   
 
An alternative definition of 
$\sup ^* S $ and $ \inf T^*$
can be given by using the surreal limit introduced in 
 Mez\H{o} \cite{M}  
for sequences of length $ \omega$ and in \cite{L} 
for sequences indexed by an arbitrary 
linearly ordered set.
If  $S$ has no maximum  and the elements of $S$ are ordered 
as a strictly increasing sequence
 $( s_i) _{i \in I} $, then  
 $\sup ^* S= \slim _{i \in I}  s_i$ in the notation of 
\cite{L}. 

Notice that Definition \ref{supast} makes sense also in case
$S$ is the empty set. In this case, $\sup ^* S$ is the empty sequence, the 
surreal of length $0$.

\begin{proposition} \labbel{s}
Suppose that $S$ is a set of surreals,  
 $s= \sup ^* S$ and 
$\ell(s) = \alpha $. Then 
$u _{ \restriction  \alpha } < s  $,
for every $u \in S$.  
In particular, $S < s$.
 \end{proposition}  

\begin{proof}
The last statement is immediate from the previous one, by Lemma \ref{lem}(3),
since  $ s = s _{ \restriction \alpha  }  $, by assumption.

If $S$ has a maximum $w$, then $w _{ \restriction \alpha } = w <s  $, by construction.
If $u \in S$, then $u\leq w $, hence,
by Lemma \ref{lem}(2),  $ u _{ \restriction \alpha } \leq w _{ \restriction \alpha }  < s $.     

Suppose that $S$ has no maximum. First we show the following
statement.

(*) $ u _{ \restriction  \alpha } \leq s  $, for every $u \in S$. 

Indeed, for every  $\gamma < \alpha $,   we have
$ w _{ \restriction \leq \gamma }  = s _{ \restriction \leq \gamma }   $,
for all sufficiently large $w \in S$,
by the definition of $s= \sup ^* S$. 
 In particular,  
we can choose    $w \geq u$, hence 
 $ u _{ \restriction \leq \gamma }  \leq 
 w _{ \restriction \leq \gamma }  =  s _{ \restriction \leq \gamma }   $.
Since this holds for every $\gamma < \alpha $, we have
$ u _{ \restriction \alpha } \leq s $,
by Lemma \ref{lem}(4) with
$u _{ \restriction \alpha }  $ in place of $u$
and since $(u _{ \restriction \alpha }) _{ \restriction \leq \gamma } 
 = u  _{ \restriction \leq \gamma }  $, if $\gamma< \alpha $. 
 
Having proved (*), suppose by contradiction that 
$ u _{ \restriction \alpha } =s  $, for some $u \in S$.
If $w \in S$ and  $w \geq u$, then  
$ s \geq w _{ \restriction  \alpha }  \geq 
u _{ \restriction \alpha  }  =s $,
by (*) with  $w$
in place of $u$  and by Lemma \ref{lem}(2), hence
 $ s = w _{ \restriction  \alpha }  = 
u _{ \restriction  \alpha  }  $.
Now $w ( \alpha )$ can assume only a finite number of values
($+$, $-$ and  possibly $undefined$).
Since $w _{ \restriction \alpha }  $ is constant,
for $w \geq u$, 
then, for $w$ sufficiently large, the value of    
$w ( \alpha )$  stabilizes, as $w$ increases.
But then, by the definition of $s= \sup ^* S$, 
 $w ( \alpha )$ should  stabilize to
the value of $s( \alpha )$, which is undefined,
by the assumption $\ell (s) = \alpha $. 
 Then  $s= w$, for some sufficiently large
$w \in S$,  hence such a $w$  would be the maximum of $S$,
contrary to the assumption that $S$ has no maximum.
 We have reached a contradiction, hence 
$u _{ \restriction  \alpha } < s  $.
\end{proof}

\begin{theorem} \labbel{1b}
Suppose that  $S$ is a set of surreals,  
 $s= \sup ^* S$ and 
$\ell(s) = \alpha $.
Then, for every surreal
$z$, 
  \begin{enumerate}
    \item[]   
$S < z$ if and only if  
$s  \leq  z _{ \restriction  \alpha }   $.
\end{enumerate}  
 \end{theorem} 

\begin{proof}
Suppose that $S < z$
and  $S$ has a maximum.
Say, $u$ is the maximum of $S$.  Then
$u < z$. Let $\gamma$ be smallest ordinal such that 
$u( \gamma )  < z ( \gamma )$, hence $\gamma < \alpha $,
by the second statement in Lemma \ref{lem}(1), since $\ell(u) < \alpha $.
Indeed, if $\ell(u)= \beta $, then Definition \ref{supast} 
gives $\alpha= \ell (s) = \beta +1$.   
If $\gamma < \beta $, then, by the definition 
of $s= \sup ^* S$, we have 
$ s( \gamma ) =  u( \gamma )  < z ( \gamma )$, hence
 $s = s _{ \restriction \alpha }   <  z_{ \restriction  \alpha }   $, by 
Lemma \ref{lem}(3).
If $\gamma  = \beta  $, then
$ u _{ \restriction \gamma } = z _{ \restriction \gamma }    $,
by Lemma \ref{lem}(1); moreover,   
$u(\gamma)$ is undefined, hence 
$z(\gamma) = +$, since $u( \gamma )  < z ( \gamma )$.
Thus 
$ z _{ \restriction \alpha } = s  $. 
In any case, $s  \leq z_{ \restriction  \alpha }   $.

Suppose that $S < z$
and $S$ has no maximum.
By the definition of $ s = \sup^* S$, we have that,
for every $\gamma < \alpha $,
there is $u \in S$
such that 
$u _{ \restriction \leq \gamma } = s _{ \restriction \leq  \gamma }    $.
Since $u < z$, we get 
$s _{ \restriction \leq \gamma } = u _{ \restriction \leq  \gamma }   
\leq z _{ \restriction \leq \gamma }   $, by Lemma \ref{lem}(2).
Since this holds for every $\gamma < \alpha $,
we get 
 $ s= s _{ \restriction \alpha  }    
\leq z _{ \restriction \alpha }   $,
by Lemma \ref{lem}(4)
with $s$ in place of $u$ and $z _{ \restriction \alpha }  $ 
in place of $s$.    

Conversely, suppose that $s  \leq  z _{ \restriction  \alpha }   $.
By Proposition \ref{s}, 
$u _{ \restriction \alpha } < s = s _{ \restriction \alpha } 
\leq  z _{ \restriction  \alpha }     $, for every $u \in S$.
But then $u < z$, by Lemma \ref{lem}(3).     
 \end{proof}

Definition \ref{supast}
and Theorem \ref{1b} not only
furnish a useful ``bounding
element'' relative to some set $S$ of surreals.
They also provide     
significant information about the possible lengths of
 separating elements of some pair of sets of sureals.

\begin{corollary} \labbel{corlen}
Suppose that $S$, $T$ are sets of surreals and  $S < T$.
Let $ s= \sup ^* S $, $ t=  \inf ^* T$,
$\alpha =\ell(s)$ and 
$ \beta  =\ell(t)$. 
If  $\varepsilon \geq 
 \max \{ \alpha, \beta   \}  $
and $w$ is any surreal number, then
$w$ separates $S$ and $T$ 
if and only if   $w _{ \restriction \varepsilon }  $ separates $S$ and $T$.
 \end{corollary}

 \begin{proof} 
Immediate from Theorem \ref{1b}
and its symmetrical version.
Indeed, from, say, 
$\varepsilon \geq \alpha $,
it follows 
$( w _{ \restriction \varepsilon }) _{ \restriction \alpha } =
w _{ \restriction \alpha }      $.   
\end{proof}

\section{A separating element} \labbel{separ}

We now have at our disposal all the tools
necessary in order to  prove Theorem \ref{2}
and its improvements. 
However, we shall first discuss an example.
One might expect that 
if $S < T$,
$s=\sup ^* S $ and $ t= \inf ^* T$,
then $s<t$, or at least $s \leq t$.
Were this the case, then 
$u < s$, for every $u \in S$,
by the last statement in Proposition \ref{s}.     
Symmetrically, 
$t < v$, for every $v \in T$,  
hence Theorem \ref{2} 
 would follow, since
then both $s$ and  $t$
(in fact, any intermediate surreal) would separate $S$ and $T$. 
However, it is not necessarily the case
that $S < T$ implies
$\sup ^* S <  \inf ^* T$.
Actually, it might happen that
$\sup ^* S >  \inf ^* T$.

Indeed, for $ i < \omega$, let
$s_i = ++ \stackrel{i} {\dots}+$ be a sequence consisting
of $i$ pluses,  
and let $S = \{ s_i \mid i \in \omega  \} $, thus 
$s= \sup ^* S =  +++ \stackrel{\omega} {\dots}$ \ 
is a sequence of $ \omega$ pluses.
For $ i < \omega$, let
$t_i = ++ \stackrel{ \omega  } {\dots} \ -- \stackrel{i} {\dots}-$ be  
a sequence of $ \omega$ pluses followed by $i$ minuses. 
If $T = \{ t_i \mid i \in \omega  \} $, then
$t = \inf ^* T =  ++ \stackrel{ \omega  } {\dots} \ 
-- \stackrel{ \omega   } {\dots}$ is
the sequence consisting of $ \omega$ pluses followed 
by $ \omega$ minuses.  
 The above example shows that it might happen that
$S < T$ and,
nevertheless, 
$\sup ^* S >\inf ^* T$. 
However, by using Theorem \ref{1b},
we can easily show that there exists some element 
separating $S$ and $T$, actually, 
either $s$ or  $t$ works (see Corollary \ref{corult} below).
But it is not necessarily the case that \emph{both}
$s$ and  $t$ work.  
In the above example, $t$ separates $S$ and $T$,
but $s$ does not separate $S$ and $T$. To the opposite! We have 
$t_i < s$, for every nonzero $ i < \omega$,
rather than $s < t_i$.

However, we can show that some prolongment
of $s$ (and, symmetrically, some prolongment of $t$)
does separate $S$ and $T$.

\begin{theorem} \labbel{ult}
If $S$ and $T$ are sets of surreals such that  $S < T$,
then some prolongment of  $\sup^* S$ separates $S$ and $T$.     
 \end{theorem}

 \begin{proof}
Let $s= \sup^* S$ and $\alpha = \ell(s)$. 
By assumption, 
$S < v$, for every $v \in T$,
hence   
$s \leq v _{ \restriction  \alpha }  $,
for every  $v \in T$, by Theorem \ref{1b}. 
Choose some ordinal $\eta$ such that 
$\eta > \ell (v)$, for every $v \in T$ 
and $\eta \geq \alpha = \ell (s)$.
Let   $ \hat{ \hat{ s }}$ be the prolongment of $s$
of length $\eta$  
obtained by adding only minuses.

Let $v \in T$. 
We have 
$ \hat{ \hat{ s }} \neq v$,
since by construction they have different length.  
If $\gamma$ is the first ordinal such that 
$ \hat{ \hat{ s }} ( \gamma )  \neq v ( \gamma )$,
then
$ \hat{ \hat{ s }} ( \gamma )  < v ( \gamma )$; this is 
trivial by construction if $\gamma \geq \alpha $
and follows from 
$s \leq v _{ \restriction  \alpha }  $
if $\gamma < \alpha $.
Hence 
$ \hat{ \hat{ s }} <  v $.  

Since, by construction,
$ \hat{ \hat{ s }} _{ \restriction \alpha } = s  $,
we have 
$u < \hat{ \hat{ s }} $, for every $u \in S$,
 by Theorem \ref{1b}.
Hence 
$ \hat{ \hat{ s }}$ separates $S$ and $T$.   
 \end{proof}

Using  Corollary \ref{corlen}, we can dramatically
cut down the length either of 
the element constructed in the proof 
of \ref{ult}, or of the element constructed in the symmetrical way.

\begin{corollary} \labbel{corult}
Suppose that $S$ and $T$ are sets of surreals and $S < T$.
If $s= \sup^* S$ and $t= \inf^* T$,
then either $s$ or $t$ separates $S$ and $T$.     

In fact, if $\ell(s)= \ell (t)$
then both $s$ and  $t$
separate $S$ and $T$.
If $\ell(s) \neq \ell (t)$, then the longer
one between $s$ and $t$
is a separating element.
 \end{corollary}

\begin{proof}
Suppose that $ \alpha = \ell(s) \geq \ell (t) = \beta $.
The prolongment $ \hat{ \hat{ s }}$ 
of $ s $ constructed 
in the proof of Theorem \ref{ult}
separates  $S$ and $T$.
By Corollary \ref{corlen},
 $ \hat{ \hat{ s }} _{ \restriction \alpha }  $
separates  $S$ and $T$, since, by assumption,
$\alpha \geq \beta $.
By construction, $ \hat{ \hat{ s }} _{ \restriction \alpha }  $ is $s$.

Symmetrically, 
if $\ell(t) \geq \ell (s) $,
then $t$ separates $S$ and $T$.  
\end{proof}

It is evident from Corollary \ref{corult} that the separating element 
given by the proof of Theorem \ref{ult}
might be longer, in general, than the shortest possible
separating element. 
As another example, if all the elements of $S$ are negative surreals
and all the elements of $T$ are positive surreals, 
then $0$ is a separating element (hence it is the shortest one!),
 but the proof of \ref{ult} furnishes a surreal longer 
than every member of $T$---in principle, 
a surreal as long as any prescribed ordinal.

Hence it is quite surprising
to discover that if we combine the proof of Theorem \ref{ult}
with the symmetric argument, we are quickly
led to the shortest separating surreal.
Actually, as we mentioned, the construction proceeds plainly
without
divisions into cases. To the contrary,
if  we work out an explicit
description of the shortest separating element,
then four cases emerge.
See Definition \ref{sep}
and Theorem \ref{boh}.

\begin{theorem} \labbel{thmthm} 
Suppose that $S$ and $T$ are sets of surreals and  $S < T$.
Let $ s= \sup ^* S $, $ t=  \inf ^* T$,
$\alpha =\ell(s)$,
$ \beta  =\ell(t)$ 
and  $\varepsilon =
 \max \{ \alpha, \beta   \} $. Let
 $ \hat{ s }$ be the  prolongment  of $s$ to 
length $\varepsilon+1$ obtained by 
  adding only minuses and,
symmetrically, let 
$ \hat{ t }$ be the  prolongment  of $t$ to 
length $\varepsilon +1$ obtained by 
  adding only
pluses.

Then  both $ \hat{ s }$ and 
$ \hat{ t }$ separate $S$ and  $T$.
Moreover,
$ \hat{ s } < \hat{ t }$.

If $\gamma$ is the shortest ordinal
such that $ \hat{ s }( \gamma ) \neq  \hat{ t }( \gamma )$,
then $ \hat{ s } _{ \restriction  \gamma }  =  \hat{ t } _{ \restriction \gamma }  $
is the shortest surreal separating $S$ and $T$.
Any other separating surreal
 is a prolongment of  $\hat{ s } _{ \restriction  \gamma }$.
\end{theorem}   

\begin{proof}
Consider the surreal $\hat{ \hat{ s }}$
constructed in the proof of Theorem \ref{ult},
taking $\eta > \varepsilon $.
It might  happen that $\hat{ \hat{ s }}$
is longer than $ \hat{ s }$; however, since $ \varepsilon+1 \geq \alpha$,
we get from Corollary \ref{corlen} that   
$ \hat{ s }$ still separates $S$ and $T$.
This is the same argument as the one in the proof
of Corollary \ref{corult}  
(so far, we could have done with prolongments of length $\varepsilon$,
rather than $\varepsilon+1$.
Adding $1$ simplifies 
the subsequent argument).
Symmetrically, $ \hat{ t }$ separates $S$ and $T$. 

By Theorem \ref{1b},
we have
$ \hat{ s } _{ \restriction \alpha } =  s  \leq \hat{ t } _{ \restriction \alpha }  $ and,
symmetrically,
$ \hat{ s } _{ \restriction \beta } \leq t = \hat{ t } _{ \restriction \beta }     $.
Hence, if 
$ \delta  = \min \{ \alpha, \beta   \}  $, then
 $ \hat{ s } _{ \restriction \delta }  \leq  \hat{ t } _{ \restriction \delta }  $,
by Lemma \ref{lem}(2).
Moreover, $ \hat{ s } \neq \hat{ t }$,
since the last element of $s$ is $-$
and the last element  of $t$ is  $+$
(here we are using the assumption that 
the prologments $ \hat{ s }$ and $ \hat{ t }$ have length
$> \max \{ \alpha, \beta   \}$).
Let $\gamma$ be the smallest ordinal
such that $ \hat{ s }( \gamma ) \neq  \hat{ t }( \gamma )$.
If $\gamma < \delta $, then
$ \hat{ s }( \gamma ) <  \hat{ t }( \gamma )$, since
 $ \hat{ s } _{ \restriction \delta }  \leq  \hat{ t } _{ \restriction \delta }  $.
If $\gamma \geq  \delta $, then,
by construction,
either 
$\hat s(\gamma) = -$
or $ \hat t(\gamma) =+$,
hence 
$ \hat{ s }( \gamma ) <  \hat{ t }( \gamma )$
in this case, as well.
In both cases, we get $ \hat{ s }< \hat{ t }$.

By Lemma \ref{lem}(1) and  the above paragraph we have 
$ \hat{ s } \leq  \hat{ s } _{ \restriction \gamma  }  =
 \hat{ t } _{ \restriction \gamma  } \leq   \hat{ t }   $.\footnote{Actually,
 we have $ \hat{ s } <   \hat{ s } _{ \restriction \gamma  }  =
 \hat{ t } _{ \restriction \gamma  } <  \hat{ t }   $,
since $\gamma \leq \varepsilon $, hence both 
$ \hat{ s }( \gamma )$ and $ \hat{ t }( \gamma ) $ are defined,
hence $ \hat{ s }( \gamma ) = -$ and $ \hat{ t }( \gamma ) =+$,
since they are distinct.
However we do not need the strict inequality in the proof.}
Since both $ \hat{ s }$ and $ \hat{ t }$
separate $S$ and $T$, then  
$ \hat{ s } _{ \restriction \gamma  } $
separates $S$ and $T$. 
Suppose that $u$ is any other separating element.
By Theorem \ref{1b},
$s \leq u _{ \restriction  \alpha }  $.
Since the prolongment $ \hat{ s }$ is obtained by adding minuses
to $s$, we also have    
$ \hat{ s }  \leq u _{ \restriction  \varepsilon  }  $.
Symmetrically, 
$  u _{ \restriction  \varepsilon  } \leq  \hat{ t }  $.
By Lemma \ref{lem}(2) and since 
$\gamma \leq \varepsilon $, we get  
$ \hat{ s } _{ \restriction  \gamma }    \leq
 u _{ \restriction  \gamma   }  \leq  \hat{ t } _{ \restriction  \gamma }   $, 
but 
$ \hat{ s } _{ \restriction  \gamma }    =  \hat{ t } _{ \restriction  \gamma }   $,
hence
$  u _{ \restriction  \gamma   } = \hat{ s } _{ \restriction  \gamma }    $,
that is, 
$  u $ is a prolongment of $  \hat{ s } _{ \restriction  \gamma }    $.
\end{proof}

It is now not difficult to evaluate explicitly the 
shortest surreal separating $S$ and $T$.

\begin{definition} \labbel{sep}
 If $(s, t)$ is an ordered pair of  surreals, 
the (ordered) \emph{separator} $\sep(s,t)$ 
of $s$ and  $t$ is defined in the following way.   

(a) If $s=t$, we let  $\sep(s,t) =s=t$.

(b) If there is some $\gamma$ 
such that $s(\gamma) \neq t( \gamma )$ are both defined, choose
$\gamma$  minimal and let $\sep(s,t) = s _{ \restriction \gamma } =
 t _{ \restriction \gamma } $. 
Notice that $s _{ \restriction \gamma } =
 t _{ \restriction \gamma } $ by Lemma \ref{lem}(1).
 
(c) If $t$ is a (proper) prolongment  
of $s$, say, $\ell (s) = \alpha $,
consider the first ordinal $\varepsilon \geq \alpha $
such that $t (\varepsilon) = +$ and let
  $\sep(s,t) = t _{ \restriction \varepsilon }$.
If no such $\varepsilon$ exists,   
that is, $t (\varepsilon) = -$, for every $\varepsilon \geq \alpha $,
then let  
 $\sep(s,t) = t $.

(d) Symmetrically, 
if $s$ is a prolongment  
of $t$ with  $\ell (t) = \beta  $,
let
  $\sep(s,t) = s _{ \restriction \varepsilon }$,
for the least  $\varepsilon \geq \beta  $
such that $s( \varepsilon ) = -$, if such a $\varepsilon$ exists,  
and let  $\sep(s,t) = s $ otherwise.

Notice that, because of Lemma \ref{lem}(1),
the above definition covers all possible cases, that is,
 $\sep(s,t)$ is defined for every pair $(s,t)$ of surreals. 
 \end{definition}   

Notice that, due to the last two clauses,
it might happen that 
$\sep(s,t) \neq \sep(t,s)$.
For example, 
if $s=+$ and $t=+-+$,
then   
$\sep(s,t) = +-$,
while
$\sep(t,s) = +$.

\begin{theorem} \labbel{boh}
If $S < T$, then
$w =\sep(\sup ^* S, \inf ^* T)$ 
 is the shortest surreal separating $S$ and $T$.
 Moreover, any $u$ 
 separating $S$ and $T$ is a prolongment of $w$.
 \end{theorem}

\begin{proof} 
Just follow the proof of \ref{thmthm}.

If $s=t$, then $ \hat{ s }$ and $ \hat{ t }$, as constructed in
the proof of  \ref{thmthm},
differ only by the last element, that is, $\gamma= \varepsilon $ 
and the shortest separating element turns out to be $s=t$.

If $\gamma <  \delta =  \inf \{ \alpha , \beta \} $, then case (b)
in Definition \ref{sep} applies.

If $\alpha \leq \beta $ and $t$ is a prolongment of $s$,
then the proof of Theorem \ref{thmthm} asks for the smallest $\gamma$ 
such that $ \hat{ s }( \gamma )$ and $ \hat{ t }( \gamma ) $ differ.
This is the first ordinal $\gamma \geq \alpha $ such that 
 $ \hat{ t }( \gamma ) =+$, since $ \hat{ s }$ assumes always the value $-$
above $\alpha$, $ \hat{ s }$ and $ \hat{ t }$ have equal length
and they are identical up to $\alpha$. 
The present case incorporates the case $s=t$, which we have treated separately
for clarity. 

The case when
 $\alpha \geq \beta $ and $s$ is  a prolongment of $t$
is symmetrical.
\end{proof}

\section{Remarks} \labbel{rmksec} 

(a) The surreal 
$\sup^* S$ is not the only surreal 
for which Theorem \ref{1b} is satisfied. 
 In fact, every prolongment $s'$ of  $\sup^* S$
obtained by adding minuses will obviously satisfy
Theorem \ref{1b}, in case  $\alpha$ there is taken to be the length
of $s'$. 
However, $\sup^* S$ is the shortest surreal for which
Theorem \ref{1b} holds. 

(b) Not every surreal $w$ can be obtained as $\sup^* S$,
for some set $S$ of surreals.
In fact, if  $w$ has some ``tail'' consisting only of minuses,
then $w$ cannot be obtained as $\sup^* S$, for some $S$.
Actually, this is an
if and only if condition. If $w$ has  $+$ as  a last element,
then $w= u \conc +$ for some surreal $u$, hence $w$
has the form    $\sup^* S$, for $S= \{ u \} $. 
If $w$ 
 has limit length and is not eventually $-$, 
then $w = \sup^* S$, where 
$S$ is the set of all initial segments of $w$
which are ``cut just below some $+$''.

(c) Remarks (a) and (b) above show that, anyway,
for every surreal $s$ there is some set 
$S$ such that the conclusion of Theorem \ref{1b} holds.

  (d)
Under the assumptions 
in Theorem \ref{1b}, we have that,   for every surreal
$z$, the following conditions are equivalent.
 \begin{enumerate}  
 \item
$S \leq  z$,
\item
either 
$s  \leq  z _{ \restriction  \alpha }   $,
or $z$  is a maximum of $S$.  
   \end{enumerate} 

Indeed,  $S \leq z$
if and only if either $S < z$,
or $z$ is the maximum of $S$.

(e) As we mentioned shortly after Definition \ref{supast},
the definition of $\sup ^* S$ can be obtained as a special case of 
 the s-limit \cite{M,L}, when $S$ has no maximum. 

It is possible to give a uniform definition of 
$\sup^*$  which takes into account also the case in which 
$S$ has a maximum.
If
the elements of $S$ are ordered in increasing order
as $( s_i) _{i \in I} $, then  
it is easy to see that    
$\sup^* S = \slim (s_i \conc +)$.
In other words, if $S$ is a set of surreals, let 
$S \conc + = \{ s \conc + \mid s \in S\} $.
Then $\sup^* S$ is obtained by taking the s-limit
of the elements of $S \conc +$, ordered
in increasing order.  A symmetrical remark applies
to $\inf ^* T$.  

\acknowledgement{We are grateful to Mark Kortink
for detecting a misprint in a previous version of this manuscript.}

\end{document}